\documentclass{amsart}

\usepackage{amsfonts,amsmath,amssymb,amsthm,synttree}
\usepackage{relsize}
\usepackage[mathscr]{eucal}
\usepackage{bbm}
\usepackage{enumerate}

\allowdisplaybreaks[1]

\PassOptionsToPackage{usenames,dvipsnames,svgnames}{xcolor}  
\usepackage{tikz}
\usetikzlibrary{arrows,positioning,automata}
\usetikzlibrary{shapes,snakes}

\newtheorem{thrm}{Theorem}[section]
\newtheorem{lem}[thrm]{Lemma}

\newtheorem{cor}[thrm]{Corollary}
\theoremstyle{definition}
\newtheorem{definition}[thrm]{Definition}
\newtheorem{defin}[thrm]{Definition}

\numberwithin{equation}{section}

\newcommand{\labeq}[1]{\label{eq:#1}}
\newcommand{\refeq}[1]{(\ref{eq:#1})}
\newcommand{\labt}[1]{\label{thm:#1}}
\newcommand{\reft}[1]{Theorem~\ref{thm:#1}}

\newcommand{\labl}[1]{\label{lemma:#1}}
\newcommand{\refl}[1]{Lemma~\ref{lemma:#1}}
\newcommand{\labd}[1]{\label{definition:#1}}
\newcommand{\refd}[1]{Definition~\ref{definition:#1}}
\newcommand{\labc}[1]{\label{coro:#1}}

\newcommand{\labf}[1]{\label{fig:#1}}
\newcommand{\reff}[1]{Figure~\ref{fig:#1}}

\newcommand{\dimh}[1]{\hbox{$\dim_{\hbox{H}}$}\left( #1\right)}

\newcommand{\q}[1]{q_1 \cdots q_{ #1 }}

\newcommand{\NN}{\mathbb{N}_2^{\mathbb{N}}}

\newcommand{\UPAmri}{\Upsilon_{Q,A_{m,r}}^{-1}}
\newcommand{\LAAmr}{\Lambda_{A_{m,r}}(Q)}

\newcommand{\wrt}[1]{\hbox{ w.r.t. }#1}

\newcommand{\ppq}{\psi_{P,Q}}

\newcommand{\floor}[1]{\left\lfloor #1 \right\rfloor} 
 
\newcommand{\br}[1]{\left\{ #1 \right\}}

\newcommand{\pr}[1]{\left( #1 \right)}
\newcommand{\bra}[1]{\left [ #1 \right ]}

\newcommand{\NQ}{\mathscr{N}(Q)}
\newcommand{\N}[1]{\mathscr{N}( #1 )}
\newcommand{\Nk}[2]{\mathscr{N}_{#2}( #1 )} 
\newcommand{\DNQ}{\mathscr{DN}(Q)}
 
\newcommand{\RNQ}{\mathscr{RN}(Q)}
 
\newcommand{\RNk}[2]{\mathscr{RN}_{#2}( #1 )} 
\newcommand{\RDN}{\RNQ \cap \DNQ \backslash \NQ}

\newcommand{\DistQ}{\mathscr{D}_{(x_n)}(Q)}

\newcommand{\qnk}{Q_n^{(k)}}

\newcommand{\pnk}{P_n^{(k)}}
\newcommand{\QnmrI}{Q_n^{(m,r)}}

\newcommand{\NAPIQ}{\mathscr{N}^{I}(Q)}
\newcommand{\NAPIIQ}{\mathscr{N}^{II}(Q)}
\newcommand{\NAPI}[1]{\mathscr{N}^{I}( #1 )}
\newcommand{\NAPII}[1]{\mathscr{N}^{II}( #1 )}
\newcommand{\NAPIk}[2]{\mathscr{N}^{I}_{ #2 }( #1 )}

\newcommand{\NAPIIkmr}[4]{\mathscr{N}^{II}_{{ #2 }, { #3 }, { #4 } }( #1 )}
\newcommand{\NAPImr}[3]{\mathscr{N}^{I}_{{#2}, {#3}}(#1)}
\newcommand{\NAPAbI}[1]{\mathscr{N}^{I}_{Ab}(#1)}
\newcommand{\NAPAbII}[1]{\mathscr{N}^{II}_{Ab}(#1)}

\newcommand{\RNAPIQ}{\mathscr{RN}^{I}(Q)}
\newcommand{\RNAPIIQ}{\mathscr{RN}^{II}(Q)}
\newcommand{\RNAPI}[1]{\mathscr{RN}^{I}( #1 )}
\newcommand{\RNAPII}[1]{\mathscr{RN}^{II}( #1 )}

\newcommand{\one}{\mathbbm{1}}

\newcommand{\R}{\mathbb{R}}

\newcommand{\oneNSk}[1]{\one_{S}(#1)}
\newcommand{\DNAb}{\mathscr{DN}_{Ab}(Q)}

\newcommand{\blank}[1]{ }

\begin{document}

\thanks{Research of the authors was partially supported by the U.S. NSF grant DMS-0943870.}

\title{Normality of different orders for Cantor series expansions}
\author[D. Airey]{Dylan Airey}
\address[D. Airey]{
Department of Mathematics, University of Texas at Austin, 2515 Speedway, Austin, TX 78712-1202, USA}
\email{dylan.airey@utexas.edu}

\author[B. Mance]{Bill Mance}
\address[B. Mance]{Institute of Mathematics of Polish Academy of Science, \'{S}niadeckich 8, 00-656 Warsaw, Poland}
\address{Department of Mathematics, University of North Texas, General Academics Building 435, 1155 Union Circle,  \#311430, Denton, TX 76203-5017, USA}

\email{Bill.A.Mance@gmail.com}

\maketitle


\begin{abstract}
Let $S \subseteq \mathbb{N}$ have the property that for each $k \in S$ the set $(S - k) \cap \mathbb{N} \setminus S$ has asymptotic density $0$. We prove that there exists a basic sequence $Q$ where the set of numbers $Q$-normal of all orders in $S$ but not $Q$-normal of all orders not in $S$ has full Hausdorff dimension. If the function $k \mapsto  \one_S(k)$ is computable, then there exist computable examples. For example, there exists a computable basic sequence $Q$ where the set of numbers normal of all even orders and not normal of all odd orders has full Hausdorff dimension. This is in strong constrast to the $b$-ary expansions where any real number that is normal of order $k$ must also be normal of all orders between $1$ and $k-1$.

Additionally, all numbers we construct satisfy the unusual condition that block frequencies sampled along non-trivial arithmetic progressions don't converge to the expected value. This is also in strong contrast to the case of the $b$-ary expansions, but more similar to the case of the continued fraction expansion. As a corollary, the set of $Q$-normal numbers that are not normal when sampled along any non-trivial arithmetic progression has full Hausdorff dimension.
\end{abstract}

\section{Introduction}

\subsection{Normal numbers}

We recall the modern definition of a normal number.

\begin{defin}\labd{normal}
A real number $x$ is {\it normal of order $k$ in base $b$} if all blocks of digits of length $k$ in base $b$ occur with relative frequency $b^{-k}$ in the $b$-ary expansion of $x$. We denote this set by $\Nk{b}{k}$.
Moreover, $x$ is {\it simply normal in base $b$} if it is a member of $\Nk{b}{1}$ and $x$ is {\it normal in base $b$} if it is normal of order $k$ in base $b$ for all natural numbers $k$. We denote the set of normal numbers in base $b$ by 
$$
\N{b}:=\bigcap_{k \in \mathbb{N}} \Nk{b}{k}.
$$
\end{defin}

We also wish to mention one of the most fundamental and important results relating to normal numbers in base $b$. The following is due to D. D. Wall in his Ph.D. dissertation \cite{Wall}.

\begin{thrm}[D. D. Wall]\labt{wall}
A real number $x$ is normal in base $b$ if and only if the sequence $(b^nx)$ is uniformly distributed mod $1$.
\end{thrm}


It is well known that \'{E}. Borel \cite{BorelNormal} was the first mathematician to study normal numbers. In 1909 he gave the following definition.

\begin{defin}[\'{E}. Borel]\labd{normalBorel}
A real number $x$ is {\it normal in base $b$} if each of the numbers $x, bx, b^2x,\cdots$ is simply normal (in the sense of \refd{normal}), in each of the bases $b, b^2, b^3, \cdots$.
\end{defin}

\'{E}. Borel proved that Lebesgue almost every real number is normal, in the sense of \refd{normalBorel}, in all bases. In 1940, S. S. Pillai \cite{Pillai2} simplified \refd{normalBorel} by proving that
\begin{thrm}[S. S. Pillai]\labt{Pillai}
For $b \geq 2$, a real number $x$ is normal in base $b$ if and only if it is simply normal in each of the bases $b, b^2, b^3, \cdots$.
\end{thrm}
\reft{Pillai} was improved in 1951 by I. Niven and H. S. Zuckerman \cite{NivenZuckerman} who proved
\begin{thrm}[I. Niven and H. S. Zuckerman]\labt{NivenZuckerman}
\refd{normal} and \refd{normalBorel} are equivalent.
\end{thrm}

A simpler proof of \reft{Pillai} was given by J. E. Maxfield in \cite{MaxfieldPillai}. J. W. S. Cassels gave a shorter proof of \reft{NivenZuckerman} in \cite{CasselsNZ}.
 It should be noted that both of these results require some work to establish, but were assumed without proof by several authors. For example, M. W. Sierpinski assumed \reft{Pillai} in \cite{Sierpinski} without proof. Moreover, D. G. Champernowne \cite{Champernowne}, A. H. Copeland and P. Erd\H os \cite{CopeErd}, and other authors took \refd{normal} as the definition of a normal number before it was proven that \refd{normal} and \refd{normalBorel} are equivalent. More information can be found in Chapter 4 of the book of Y. Bugeaud \cite{BugeaudBook}.

The following theorem was proven by H. Furstenberg in his seminal paper ``Disjointness in Ergodic Theory, Minimal Sets, and a Problem in Diophantine Approximation'' \cite{FurstenbergDisjoint} on page 23 as an application of disjointness to stochastic sequences.
\begin{thrm}[H. Furstenberg]\labt{basebnormalii}
Suppose that $x=d_0.d_1d_2\cdots$ is the $b$-ary expansion of $x$. Then $x$ is normal in base $b$ if and only if for all natural numbers $m$ and $r$ the real number $0.d_{r}d_{m+r}d_{2m+r}d_{3m+r}\cdots$ is normal in base $b$.
\end{thrm}

It is interesting to note that although Furstenberg did not provide an alternate
proof of \reft{NivenZuckerman}, he showed that an entirely different definition of
normality is equivalent to \refd{normal}. 
We will say that $x$ is \textit{AP normal of type I in base $b$} if $x$ satisfies \refd{normalBorel} and \textit{AP normal of type II in base $b$} if $x$ satisfies the notion introduced in \reft{basebnormalii}. Thus, for numbers expressed in base $b$
\begin{equation}\labeq{APequiv}
\hbox{normality }\Leftrightarrow \hbox{ AP normality of type I } \Leftrightarrow \hbox{ AP normality of type II}.
\end{equation}
The core of \'{E}. Borel's definition is that a number is normal in base $b$ if blocks of digits occur with the desired relative frequency along all infinite arithmetic progressions. Similarly, the core of H. Furstenberg's definition deals with testing for blocks in a different way along arithmetic progressions. 

The authors feel that the equivalence of \refd{normal} and \refd{normalBorel} and other similar ones is a far more delicate topic than is typically assumed. Recent papers \cite{ppq2} and \cite{AireyMance1} establish that the connection between analagous notions is weaker and more subtle for Cantor series expansions. Furthermore, the situation is far worse for the continued fraction expansion. Let $[a_1,a_2,a_3,\ldots]$ be normal with respect to the continued fraction expansion. B. Heersink and J. Vandehey \cite{HeersinkVandehey} recently proved that for any integers $m \geq 2, k \geq 1$, the continued fraction $[a_k,a_{m+k},a_{2m+k},a_{3m+k},\ldots]$ is never normal with respect to the continued fraction expansion. One of the main goals of this paper will be to greatly strengthen the results in \cite{ppq2}.


The relationship between numbers normal of order $r$ and $s$ is straightforward. It is easy to show that if a real number $x$ is normal of order $r>s$ in base $b$, then it must also be normal of order $s$ in base $b$. Thus, for all $k$
$$
\Nk{b}{k} \subseteq \bigcap_{j \in \{1,\ldots,k-1\}} \Nk{b}{j}.
$$
We remark that the related concept of simple normality in base $b^k$ has been studied, but is different than normality of order $k$. Y. Bugeaud proved the following theorem in \cite{BugeaudSimple}.
\begin{thrm}\labt{BugeaudSimple}
Let $a$ and $c$ be integers $\geq 2$ such that $a$ is not an integer power of $c$. Then the set of real numbers which are simply normal in base $a$ but not base $c$ has full Hausdorff dimension.
\end{thrm}
Taking $a = b^k$ and $c = b^{k+1}$ shows $\dim_H(\Nk{b^k}{1} \setminus \Nk{b^{k+1}}{1}) = 1$. 
We will be interested in difference sets of numbers normal of different orders for the Cantor series expansions. Thus, we record what is known for $b$-ary expansions. The notes in section 1.8 of \cite{KuN} describes the results of many papers on this subject. However, the authors are unaware of a version of \reft{BugeaudSimple} for normality of order $k$. Thus, we will later prove as a consequence of a theorem of C. Colebrook \cite{Colebrook} that
\begin{thrm}\labt{diffbary}
The set of real numbers which are normal of order $k$ but not order $k+1$ in base $b$ has full Hausdorff dimension. That is 
$$
\dimh{\Nk{b}{k} \backslash \Nk{b}{k+1}}=1.
$$
\end{thrm}

We also mention that the complexity
\footnote{An introduction to descriptive set theory is beyond the scope of this paper. Thus, the interested reader is referred to the book ``Classical descriptive set theory'' by A. Kechris \cite{Kechris}.}
 of such difference sets has recently been studied.
K. Beros considered sets involving normal numbers in the difference heirarchy in \cite{BerosDifferenceSet}. He proved that for $b\geq 2$ and $s>r\geq 1$, the set $\Nk{b}{r} \backslash \Nk{b}{s}$ is $\mathcal{D}_2(\Pi_3^0)$-complete. Additionally, the set $\bigcup_k \Nk{2}{2k+1} \backslash \Nk{2}{2k+2}$ is shown to be $\mathcal{D}_\omega(\Pi_3^0)$-complete.

\subsection{Normality with respect to the Cantor series expansions}

The study of normal numbers and other statistical properties of real numbers with respect to large classes of Cantor series expansions was  first done by P. Erd\H{o}s and A. R\'{e}nyi in \cite{ErdosRenyiConvergent} and \cite{ErdosRenyiFurther} and by A. R\'{e}nyi in \cite{RenyiProbability}, \cite{Renyi}, and \cite{RenyiSurvey} and by P. Tur\'{a}n in \cite{Turan}.

The $Q$-Cantor series expansions, first studied by G. Cantor in \cite{Cantor},
are a natural generalization of the $b$-ary expansions. G. Cantor's motivation to study the Cantor series expansions was to extend the well known proof of the irrationality of the number $e=\sum 1/n!$ to a larger class of numbers. Results along these lines may be found in the monograph of J. Galambos \cite{Galambos}. 
Let $\mathbb{N}_k:=\mathbb{Z} \cap [k,\infty)$. If $Q \in \NN$, then we say that $Q$ is a {\it basic sequence}.
Given a basic sequence $Q=(q_n)_{n=1}^{\infty}$, the {\it $Q$-Cantor series expansion} of a real number $x$  is the (unique)\footnote{Uniqueness can be proven in the same way as for the $b$-ary expansions.} expansion of the form
\begin{equation} \labeq{cseries}
x=E_0+\sum_{n=1}^{\infty} \frac {E_n} {q_1 q_2 \cdots q_n}
\end{equation}
where $E_0=\floor{x}$ and $E_n$ is in $\{0,1,\ldots,q_n-1\}$ for $n\geq 1$ with $E_n \neq q_n-1$ infinitely often. We abbreviate \refeq{cseries} with the notation $x=E_0.E_1E_2E_3\ldots$ w.r.t. $Q$.


Let
$$
Q_n^{(k)}:=\sum_{j=1}^n \frac {1} {q_j q_{j+1} \cdots q_{j+k-1}} \hbox{ and }  T_{Q,n}(x):=\left(\prod_{j=1}^n q_j\right) x \bmod{1}.
$$
A. R\'enyi \cite{Renyi} defined a real number $x$ to be {\it normal} with respect to $Q$ if for all blocks $B$ of length $1$,
\begin{equation}\labeq{rnormal}
\lim_{n \rightarrow \infty} \frac {N_n^Q (B,x)} {Q_n^{(1)}}=1.
\end{equation}
If $q_n=b$ for all $n$ and we restrict $B$ to consist of only digits less than $b$, then \refeq{rnormal} is equivalent to {\it simple normality in base $b$}, but not equivalent to {\it normality in base $b$}. 
A basic sequence $Q$ is {\it $k$-divergent} if
$\lim_{n \rightarrow \infty} Q_n^{(k)}=\infty$,  {\it fully divergent} if $Q$ is $k$-divergent for all $k$, and {\it $k$-convergent} if it is not $k$-divergent. 
A basic sequence $Q$ is {\it infinite in limit} if $q_n \rightarrow \infty$.

Motivated by \reft{wall}, we make the following definitions of normality for Cantor series expansions.

\begin{definition} A real number $x$  is {\it $Q$-normal of order $k$} if for all blocks $B$ of length $k$,
$$
\lim_{n \rightarrow \infty} \frac {N_n^Q (B,x)} {Q_n^{(k)}}=1.
$$
We let $\Nk{Q}{k}$ be the set of numbers that are $Q$-normal of order $k$. The real number $x$ is {\it $Q$-normal} if
$x \in \NQ := \bigcap_{k=1}^{\infty} \Nk{Q}{k}.$
\end{definition}
\begin{definition}
A real number $x$ is {\it $Q$-ratio normal of order $k$} (here we write $x \in \RNk{Q}{k}$) if for all blocks $B_1$ and $B_2$ of length $k$
$$
\lim_{n \to \infty} \frac {N_n^Q (B_1,x)} {N_n^Q (B_2,x)}=1.
$$
We say that $x$ is {\it $Q$-ratio normal} if
$
x \in \RNQ := \bigcap_{k=1}^{\infty} \RNk{Q}{k}.
$
\end{definition}
\begin{definition}
A real number~$x$ is {\it $Q$-distribution normal} if
the sequence $(T_{Q,n}(x))_{n=0}^\infty$ is uniformly distributed mod $1$. Let $\DNQ$ be the set of $Q$-distribution normal numbers.
\end{definition}

We note that by \reft{wall}, the analogous versions of the above definitions are equivalent for the $b$-ary expansions.

It was proven in \cite{ppq1} that the directed graph in \reff{figure1} gives the complete containment relationships between these notions when $Q$ is infinite in limit and fully divergent. The vertices are labeled with all possible intersections of one, two, or three choices of the sets $\NQ$, $\RNQ$, and $\DNQ$, where we know that $\NQ=\NQ \cap \RNQ$ and $\NQ \cap \DNQ=\NQ \cap \DNQ \cap \RNQ$. The set labeled on vertex $A$ is a subset of the set labeled on vertex $B$ if and only if there is a directed path from $A$ to $B$.  For example, $\NQ \cap \DNQ \subseteq \RNQ$, so all numbers that are $Q$-normal and $Q$-distribution normal are also $Q$-ratio normal. 

We remark that all inclusions suggested from \reff{figure1} are either easily proven ($\NQ \subseteq \RNQ$) or are trivial. The difficulty comes in showing a lack of inclusion. The most challenging of these is to prove that there is a basic sequence $Q$ where $\RDN\neq \emptyset$.




\begin{figure}
\caption{}
\labf{figure1}
\begin{tikzpicture}[>=stealth',shorten >=1pt,node distance=3.4cm,on grid,initial/.style    ={}]
  \node[state]          (NQ)                        {$\mathsmaller{\NQ}$};
  \node[state]          (RNQ) [above right =of NQ]    {$\mathsmaller{\RNQ}$};
  \node[state]          (RNQDNQ) [below right=of RNQ]    {$\mathsmaller{\RNQ \cap \DNQ}$};
  \node[state]          (NQDNQ) [below right =of NQ]    {$\mathsmaller{\NQ \cap \DNQ}$};
  \node[state]          (DNQ) [above right=of RNQDNQ]    {$\mathsmaller{\DNQ}$};
\tikzset{mystyle/.style={->,double=black}} 
\tikzset{every node/.style={fill=white}} 
\path (RNQDNQ)     edge [mystyle]    (RNQ)
      (RNQDNQ)     edge [mystyle]     (DNQ)
      (NQ)     edge [mystyle]     (RNQ)
      (NQDNQ)     edge [mystyle]     (RNQDNQ)
      (NQDNQ)     edge [mystyle]     (NQ);
\tikzset{mystyle/.style={<->,double=black}}
\end{tikzpicture}
\end{figure}

It  follows from a well known result of H. Weyl \cite{Weyl2,Weyl4} that $\DNQ$ is a set of full Lebesgue measure for every basic sequence $Q$. We will need the following result of the second author \cite{Mance4} later in this paper.

\begin{thrm}\labt{measure}\footnote{Early work in this direction has been done by A. R\'enyi \cite{Renyi}, T. \u{S}al\'at \cite{Salat4}, and F. Schweiger~\cite{SchweigerCantor}.}
Suppose that $Q$ is infinite in limit. Then $\Nk{Q}{k}$ and $\RNk{Q}{k}$ are of full measure if and only if $Q$ is $k$-divergent. The sets $\NQ$ and $\RNQ$ are of full measure if and only if $Q$ is fully divergent. 
\end{thrm}

A surprising property of $Q$-normality of order $k$ is that we may not conclude that $\Nk{Q}{k} \subseteq \Nk{Q}{j}$ for all $j <k$ like we may for the $b$-ary expansions. In fact, it was shown in \cite{ppq2} that for every $k$ there exists a basic sequence $Q$  such that $\Nk{Q}{k} \backslash \bigcup_{j=1}^{k-1} \Nk{Q}{j}$ is non-empty. 
Thus, we will have to be more careful in stating exactly what our theorems prove since lack of $Q$-normality of order $2$ does not imply lack of $Q$-normality of order $338$, for example. In \cite{AireyManceHDDifference} for each natural number $\ell$ a class of basic sequences were constructed with the property that
$$
\dimh{\bigcap_{j=\ell}^\infty \Nk{Q}{j} \Big \backslash \bigcup_{j=1}^{\ell-1} \Nk{Q}{j}}=1.
$$
These results will be strongly improved in this paper and we will consider more general sets of the form 
$$
\bigcap_{k \in S} \Nk{Q}{k} \setminus  \bigcup_{k \notin S} \Nk{Q}{k}.
$$
\subsection{$Q$-normality along arithmetic progressions}

In this subsection we will extend the definitions of AP normality of types I and II to certain classes of Cantor series expansions. The statements are at heart very simple, but necessarily get somewhat technical. Thus it may be helpful to keep in mind the definitions in the case of the $b$-ary expansion.

Suppose that $M=(m_t)_t$ is an increasing sequence of positive integers. Let $N_{M,n}^Q(B,x)$ be the number of occurrences of the block $B$ at positions $m_t$ for $m_t \leq n$ in the $Q$-Cantor series expansion of $x$.
 If $x=E_0.E_1E_2 \cdots$ w.r.t. {P}, then put
$$
\ppq(x):=\sum_{n=1}^\infty \frac {\min(E_n,q_n-1)} {\q{n}}.
$$
The functions $\ppq$ and their properties
\footnote{See \cite{ppq1} for an overview of properties such as continuity and multifractal analysis of $\ppq$. There are many fractals associated with the functions $\ppq$, but they will not affect the results discussed in this paper.}
 will be of critical importance to our constructions and were the topic of a predecessor to this paper by the second author \cite{ppq1}. The key property of these functions is the following theorem that was proven in \cite{ppq1}.
\begin{thrm}\labt{mainpsi}

Suppose that $M=(m_t)$ is an increasing sequence of positive integers and  $P$ and $Q$ are basic sequences which are infinite in limit.
  If $x=E_0.E_1E_2\cdots$ w.r.t $P$ satisfies
$E_n< q_{n}-1$ for infinitely many $n$, then for every block~$B$
$$
N_{M,n}^{Q}\left(B,\psi_{P,Q}(x)\right)
=N_{M,n}^{P}(B,x)+O(1).
$$
\end{thrm}

Given a basic sequence $M$, we define the basic sequence $\Lambda_M(Q):=(q_{m_t})_{t=1}^\infty$.
If $x=E_0.E_1E_2\cdots$ w.r.t. $Q$, then let 
$$
\Upsilon_{Q,M}(x):=0.E_{m_1}E_{m_2}E_{m_3}\cdots \hbox{ w.r.t. }\Lambda_M(Q).
$$ 
For $m \in \mathbb{N}$ and $0 \leq r \leq m-1$ let $\mathscr{A}_{m,r}:=(mt+r)_{t=0}^\infty$. 
Set 
\begin{align*}
N_{n,m,r}^Q(B,x)&:=N_{\mathscr{A}_{m,r},n}(B,x);\\
{N_{n,m,r}^Q}^\prime(B,x)&:=N_n^{\Lambda_{\mathscr{A}_{m,r}}(Q)}(B,\Upsilon_{Q,\mathscr{A}_{m,r}}(x));\\
Q_{n}^{(m,r)} &:= \sum_{j=0}^{\lfloor \frac {n-r}{m} \rfloor} \frac {1}{q_{mj+r}q_{mj+r+1}...q_{mj+r+m-1}}.
\end{align*}
The following definition is motivated by \refd{normalBorel} and \reft{basebnormalii}.

\begin{definition}\labd{apnormal}
Let $\NAPImr{Q}{m}{r}$ be the set of real numbers $x$ such that 
$$
\lim_{n \rightarrow \infty} \frac{N_{n,m,r}^Q(B,x)}{Q_{n}^{(m,r)}}=1
$$
for all blocks $B$ of length $m$. 
A real number $x$ is  {\it AP $Q$-normal of type I} if 
$$
x \in \NAPIQ:=\bigcap_{m=1}^\infty \bigcap_{r = 0}^{m-1} \NAPIk{Q}{m,r}.
$$
Let 
$$
\NAPIIkmr{Q}{k}{m}{r}:=\UPAmri\left( \Nk{\LAAmr}{k}  \right).
$$
We say $x$ is {\it AP $Q$-normal of type $II$} if
$$
x \in \NAPII{Q}:= \bigcap_{m = 1}^\infty \bigcap_{ r = 0}^{m-1} \bigcap_{k=1}^\infty \NAPIIkmr{Q}{k}{m}{r}.
$$
 We also say a real number $x$ is {\it AP $Q$-abnormal of type I} if $x$ is an element of
 $$\NAPAbI{Q} := \R \setminus \bigcup_{m = 2}^\infty \bigcup_{r = 0}^{m-1} \NAPImr{Q}{m}{r} $$
 and {\it AP $Q$-abnormal of type II} if $x$ is an element of
 $$
\NAPAbII{Q} := \R \setminus \bigcup_{m = 2}^\infty \bigcup_{r = 0}^{m-1}  \bigcup_{k=1}^\infty \NAPIIkmr{Q}{k}{m}{r}.
 $$
The sets\footnote{It is unknown how the sets $\RNQ$, $\RNAPIQ$, and $\RNAPIIQ$ are related except that $\RNAPIQ \subseteq \RNQ$.}
 $\RNAPIQ$ and $\RNAPIIQ$ are defined similarly to $\NAPIQ$ and $\NAPIIQ$. Note that $\NAPIQ \subseteq \RNAPIQ$ and $\NAPIIQ \subseteq \RNAPIIQ$.
\end{definition}

The sets $\NAPI{Q}$ and $\NAPII{Q}$ introduced in \refd{apnormal} give a natural extension of the notions of AP normality of type I and II given in \refd{normalBorel} and \reft{basebnormalii}.

A basic sequence $Q$ is {\it $(m,r)$-divergent of type I (resp. $(k,m,r)$-divergent of type II)} if $\lim_{n \rightarrow \infty} \QnmrI =\infty$
(resp. $\lim_{n \rightarrow \infty} (\LAAmr)_n^{(k)}=\infty$).
$Q$ is {\it fully divergent of type I} if $Q$ is $(m,r)$-divergent of type I  for all $m \geq 1$, and $0 \leq r \leq m-1$.
$Q$ is {\it fully divergent of type II} if $Q$ is $(k,m,r)$-divergent of type II  for all $k\geq 1$, $m\geq 1$, and $0 \leq r \leq m-1$.
Suppose that $Q$ is infinite in limit. It follows by Theorems $4.6$ and $4.11$ and their proofs in \cite{Mance4} that $\NAPImr{Q}{m}{r}$ (resp. $\NAPIIkmr{Q}{k}{m}{r}$) is a set of full Lebesgue measure if and only if $Q$ is $(m,r)$-divergent of type I (resp. $(k,m,r)$-divergent of type II).

\begin{definition}
 A real number $x$ is \textit{AP $Q$-distribution normal} if for all $m \in \mathbb{N}$ and $0 \leq r < m$ we have that $\pr{T_{Q,mn+r}(x)}_{n}$ is u.d. mod 1.
We say that $x$ is \textit{AP $Q$-distribution abnormal} if $\pr{T_{Q,mn+r}(x)}_{n}$ is not u.d. mod 1 for any $m>1$. We denote the set of AP $Q$-distribution abnormal numbers $\DNAb$.
\end{definition}

It was shown by the authors in \cite{AireyMance1} that $\dimh{\DNQ \cap \DNAb}=1$ whenever $Q$ is infinite in limit.
We will later study the set $\NQ \cap \NAPAbI{Q}  \setminus  \NAPII{Q}$.

\subsection{Computability}
\begin{definition}
A sequence of integers $(a_n)$ is {\it computable} if there is a total recursive function $f: \mathbb{N} \to \mathbb{N}$ such that $f(n) = a_n$.
A real number $x$ is {\it computable} if there is a total recursive function $f: \mathbb{N} \to \mathbb{N}$ such that
$$
\frac{f(n)}{n} \leq x \leq \frac{f(n)+1}{n}.
$$
\end{definition}
Informally, we may think of a computable real number as one that we may write a computer program to approximate to arbitrary precision.

A real number is {\it absolutely normal} if it is normal in base $b$ for all $b \geq 2$.
M. W. Sierpi\'{n}ski gave an example of an absolutely normal number that is not computable in \cite{Sierpinski}.
The authors feel that examples such as M. W. Sierpi\'{n}ski's are not fully explicit since they are not computable real numbers, unlike Champernowne's number. 
A. M. Turing gave the first example of a computable absolutely normal number in an unpublished manuscript. This paper may be found in his collected works \cite{Turing}. The $n$'th digit of A. M. Turing's number may be computed with an algorithm that is doubly exponential in $n$. V. Becher, P. A. Heiber, and T. A. Slaman constructed an absolutely normal number in \cite{BecherHeiberSlaman} whose digits may be computed in polynomial time.
See \cite{BecherFigueiraPicchi} by V. Becher, S. Figueira, and R. Picchi for further discussion.

P. Lafer \cite{Lafer} asked for a construction of a $Q$-distribution normal number for an arbitrary basic sequence $Q$. We remark that there are basic sequences $Q$ such that no $Q$-distribution normal number is computable. A. A. Beros and K. A. Beros showed in \cite{BerosBerosComputable} that there exists a limit computable basic sequence $Q$ such that no computable real number is $Q$-distribution normal.
Thus, it is impossible to answer P. Laffer's question for an arbitrary basic sequence. It remains open how much one needs to assume about a basic sequence $Q$ in order to guarantee existence of computable $Q$-distribution normal numbers and how to construct one of these numbers.

Note in particular that if $Q = (q_n)$ and $(E_n)$ are computable sequences then $\sum_{n=1}^\infty \frac{E_n}{q_1 \cdots q_n}$ is a computable real number.

\subsection{Statement of results}

\begin{definition}
We say a subset $S \subseteq \mathbb{N}$ is \textit{almost closed under addition} if for each $k \in S$ the set $(S - k) \cap \mathbb{N} \setminus S$ has asymptotic density $0$.
\end{definition}
Examples of such sets include the following
\begin{itemize}
\item Sets of density $0$ or $1$ such as $\mathbb{N}$, any finite set, or the primes
\item Sets for which there exists a set $C$ closed under addition such that $S \subseteq C$ and $d^*(S) = d^*(C)$.
\item $\bigcup_{k \in \mathbb{N}} \{k^2, k^2 + 1, \cdots, k^2 + k\}$
\end{itemize}
For the remainder of this paper, we will assume that all $S \subseteq \mathbb{N}$ are almost closed under addition.

We first state weaker, but less technical, versions of theorems that we will prove. These theorems will at least say some of what we can do with comparing normality of different orders as well as considering sets of numbers that are normal but not AP normal.

\begin{thrm}\labt{HDweak1}
There exists a fully divergent, infinite in limit basic sequence $Q$ such that
$$
\dimh{\bigcap_{k \in S} \Nk{Q}{k} \setminus  \bigcup_{k \notin S} \Nk{Q}{k} } = 1.
$$
If the function $k \mapsto  \one_S(k)$ is computable, then the basic sequence $Q$ can be chosen to be computable.
\end{thrm}
Thus, there exists a computable basic sequence $Q$ where the set of numbers that are $Q$-normal of all prime orders and not $Q$-normal of all composite orders has full Hausdorff dimension. Similarly, we may choose a computable $Q$ where the set of numbers that are $Q$-normal of all even orders, but not $Q$-normal of all odd orders has full Hausdorff dimension. However, \reft{HDweak1} does not guarantee that there exists a basic sequence $Q$ where there is even one real number that is $Q$-normal of all odd orders and not $Q$-normal of any even orders.\footnote{The authors strongly believe that such $Q$ exist. See the discussion in the next section.}

\begin{thrm}\labt{HDweak2}
There exists a fully divergent, infinite in limit basic sequence $Q$ such that
$$
\dimh{\NQ \cap \NAPAbI{Q}  \setminus  \NAPII{Q} } = 1.
$$
\end{thrm}

\reft{HDweak1} and \reft{HDweak2} will follow as corollaries of the following much more powerful theorems that we will prove.

\begin{thrm}\labt{main}
If $Q$ is a basic sequence which is fully divergent, non-decreasing, and infinite in limit, and $S$ is almost closed under addition then
$$
\RNAPI{Q} \cap \RNAPII{Q} \cap \NAPAbI{Q} \cap \bigcap_{k \in S} \Nk{Q}{k} \setminus \pr{\NAPII{Q} \cup \bigcup_{k \notin S} \Nk{Q}{k} } \neq \emptyset.
$$
\end{thrm}

The following follows by a slight modification of the proof of \reft{main}. 
\begin{cor}
If $Q$ is a basic sequence which is non-decreasing and infinite in limit then
$$
\RNAPI{Q}\cap \RNAPII{Q} \cap \NAPAbI{Q} \setminus \NAPII{Q} \neq \emptyset.
$$
\end{cor}

We recall the following definition from \cite{StrauchSampler}.
\begin{defin}
For a sequence $(x_n)$ of real numbers in $[0,1)$ a non-decreasing function $f : [0,1] \to [0,1]$ with $f(0) = 0$, $f(1) = 1$ is a {\it distribution function} of $(x_n)$ if there is an increasing sequence of natural numbers $N_1<N_2<N_3<\cdots$ such that for any $\xi \in [0,1]$ we have
\begin{align*}
\lim_{k \to \infty} \frac{\# \br{n \leq N_k : x_n \in [0,\xi) }}{N_k} = f(\xi).
\end{align*}
We define $G((x_n)_n)$ to be the set of distribution functions of the sequence $(x_n)$.
\end{defin}
In other terms $G((x_n)_n)$ can be identified with the set of weak-* limit points of the sequence of probability measures $\frac{1}{N} \sum_{n=1}^N \delta_{x_n}$. 
\begin{definition}
Given a sequence $(x_n)$ in $[0,1)$ define
\begin{align*}
\DistQ = \bigg \{y \in \mathbb{R} : G(T_{Q,n}(y) - x_n) = \{1 - \one_{\{0\}}\} \bigg \}.
\end{align*}
This is the set of real numbers with the property that the difference between $T_{Q,n}(y)$ and $x_n$ tends to $0$ except perhaps along a set of density $0$. Note this is a stronger condition than $G(x_n) = G(T_{Q,n}(y))$ but weaker than the condition that $\lim_{n \to \infty} T_{Q,n}(y) - x_n = 0$.
\end{definition}

\begin{thrm}\labt{HDmain}
If $S$ is almost closed under addition, then there exists a fully divergent, infinite in limit basic sequence $Q$ such that for any sequence $(x_n)$ in $[0,1)$
$$
\dim_H \pr{\DistQ \cap \RNAPI{Q} \cap \RNAPII{Q} \cap \NAPAbI{Q} \cap \bigcap_{k \in S} \Nk{Q}{k} \setminus \pr{ \NAPII{Q} \cup \bigcup_{k \notin S} \Nk{Q}{k} }} = 1.
$$
\end{thrm}
We remark that \reft{HDmain} provides a much stronger version of the main result of \cite{AlMa}. That is, it shows that there exists a basic sequence $Q$ and a real number $x$ where $x \in \NQ \setminus \DNQ$ by setting $S=\mathbb{N}$ and letting $x_n \to 0$.

Let $(x_n)$ be a uniformly distributed sequence which is not uniformly distributed along any non-trivial arithmetic subsequence. Such sequences were constructed in \cite{AireyMance1}. Then we have the following corollary.
\begin{cor}\labc{corAlMa}
There exists a fully divergent infinite in limit basic sequence $Q$ such that
\begin{align*}
\dim_H \pr{\DNQ \cap \DNAb \cap \RNAPI{Q} \cap \RNAPII{Q} \cap \NAPAbI{Q} \cap \bigcap_{k \in S} \Nk{Q}{k} \setminus \pr{\NAPII{Q} \cup \bigcup_{k \notin S} \Nk{Q}{k}}} = 1.
\end{align*}
\end{cor}

\begin{thrm}\labt{comp}
If $S$ is almost closed under addition, then there exists a computable basic sequence $Q$ which is infinite in limit such that if $k \mapsto  \one_S(k)$ is computable then there is a computable real number in 
$$
\RNAPI{Q} \cap \RNAPII{Q} \cap \NAPAbI{Q} \cap \bigcap_{k \in S} \Nk{Q}{k} \setminus \pr{\NAPII{Q} \cup \bigcup_{k \notin S} \Nk{Q}{k}}.
$$
\end{thrm}

\section{Sketch of proof and general discussion}
In this section, we wish to sketch an outline of the basic construction. We feel that many of the ideas used in the proofs will be easily obscured in computations and feel that a section informally outlining some of the ideas would be helpful. Moreover, we feel that the techniques used here might prove to be helpful when examing problems relating to other non-autonomous dynamical systems. Thus, we start with a basic problem to illustrate some of the ideas. We remark that we do not sketch how the results on Hausdorff dimension are proven. The techniques used are a combination of the ideas outlined here and those used in \cite{AireyManceHDDifference}.

We want a basic sequence $Q$ and a real number $y$ which is $Q$-normal of orders $1$ and $3$ but not $Q$-normal of order $2$. Additionally we wish for $y$ to not be $Q$-AP-normal of types $I$ or $II$. We start with a basic sequence $P$ which is infinite in limit, monotone, and slowly growing. By \cite{Mance4} there exists a real number $x \in \mathcal{N}(P) \cap \NAPI{P} \cap \NAPII{P}$. Consider $n \equiv 1 \bmod{6}$ sufficiently large so that $p_n = p_m = \alpha$ for $m$ close to $n$. The expected number of occurrences of a block of small digits $B$ of length $1$ between positions $n$ and $n+5$ in $x$ is
$$
\frac{1}{p_n} + \frac{1}{p_{n+1}} + \cdots + \frac{1}{p_{n+5}} = \frac{6}{\alpha}.
$$
Similarly if $B$ is a block of small digits of length $2$ the expected number of occurrences between $n$ and $n+5$ in $x$ is 
$$
\frac{1}{p_n p_{n+1}} + \frac{1}{p_{n+1} p_{n+2}} + \cdots + \frac{1}{p_{n+5} p_{n+6}} = \frac{6}{\alpha^2}
$$
and for blocks of length $3$ the expected number is $\frac{6}{\alpha^3}$. We define a new basic sequence $Q= (q_n)$ from $P$ by
$$
q_n = \begin{cases}
\max\{\lfloor c_{r}^{-1} p_n \rfloor, 2\} & \text{ if } n \equiv r \bmod{6} \text{ for } 1 \leq r \leq 3 \\
2^n p_n & \text{ otherwise}
\end{cases}
$$
for $c_1, c_2, c_3$ we will define soon.

Now the expected number of occurrences of a block $B$ of length $1$ in a $Q$-normal number between positions $n$ and $n+5$ is
\begin{align*}
\frac{1}{q_n} + \frac{1}{q_{n+1}} + \ldots + \frac{1}{q_{n+5}} &\approx \frac{c_1}{p_n} + \frac{c_2}{p_{n+1}} + \frac{c_3}{p_{n+2}} + \text{ lower order terms}\\
& \approx \frac{c_1 + c_2 + c_3}{\alpha}.
\end{align*}
For blocks of length $2$ the expected number of occurrences is
\begin{align*}
\frac{1}{q_n q_{n+1}} + \frac{1}{q_{n+1}q_{n+2}} + \ldots + \frac{1}{q_{n+5} q_{n+6}} &\approx \frac{c_1 c_2}{p_n p_{n+1}} + \frac{c_2 c_3}{p_{n+1} p_{n+2}} + \text{ lower order terms}\\
 &\approx \frac{c_1 c_2 + c_2 c_3}{\alpha^2}.
\end{align*}
For blocks of length $3$ the expected number is approximately
$$
\frac{c_1 c_2 c_3}{\alpha^3}.
$$

Consider $y = \Psi_{P,Q}(x)$. Then by \reft{mainpsi} $N_n^Q(B,y) = N_n^P(B,x) + O(1)$. Since $x$ is $P$-normal of all orders and both $P$ and $Q$ are fully divergent we have $y$ is $Q$-normal of order $k$ if and only if $\lim_{n \to \infty} \frac{P_n^{(k)}}{Q_n^{(k)}} = 1$. The basic sequence $P$ is slowly growing so
$$
\lim_{n \to \infty} \frac{P_n^{(1)}}{Q_n^{(1)}} = \lim_{n \to \infty} \frac{\frac{1}{p_n} + \frac{1}{p_{n+1} } + \cdots + \frac{1}{p_{n+5} }}{\frac{1}{q_n } + \frac{1}{q_{n+1} } + \cdots + \frac{1}{q_{n+5}}} = \frac{6}{c_1 + c_2 + c_3}.
$$
Similarly
\begin{align*}
\lim_{n \to \infty} \frac{P_n^{(2)}}{Q_n^{(2)}} & = \frac{6}{c_1 c_2 + c_2 c_3} \\
\lim_{n \to \infty} \frac{P_n^{(3)}}{Q_n^{(3)}} & = \frac{6}{c_1 c_2 c_3}.
\end{align*}
We thus want to solve
\begin{align*}
c_1+c_2+c_3&=6\\
c_1c_2+c_2c_3&\neq 6\\
c_1c_2c_3&=6
\end{align*}
An example of a particularly simple solution is $c_1 = 3$, $c_2 = 2$, $c_3 = 1$ which gives
\begin{align*}
c_1 + c_2 + c_3 &= 6\\
c_1c_2+c_2c_3 &= 8\\
c_1c_2c_3&=6.
\end{align*}

The expected number of occurrences of a block of length $2$ on odd positions in $x$ between positions $n$ and $n+5$ is
\begin{align*}
\frac{1}{p_n p_{n+1}} + \frac{1}{p_{n+2} p_{n+3}} + \frac{1}{p_{n+4} p_{n+5}} = \frac{3}{\alpha^2}.
\end{align*}
On the other hand the expected number of occurrences in $y$ is
\begin{align*}
\frac{1}{q_n q_{n+1}} + \frac{1}{q_{n+2} q_{n+3}} + \frac{1}{q_{n+4} q_{n+5}} \approx \frac{c_1 c_2}{\alpha^2} + \text{ lower order terms} \approx \frac{6}{\alpha^2}.
\end{align*}
By the same reasoning as above $y$ is not $Q$-AP normal of type $I$. A similar computation shows that $y$ is not $Q$-AP normal of type $II$.

Let $S \subset \mathbb{N}$ be arbitrary. In general we want to construct 
$$
y \in \bigcap_{k \in S} \Nk{Q}{k} \setminus \bigcup_{k \notin S} \Nk{Q}{k}.
$$ 
To do this the window from $n$ to $n+5$ will be replaced with a window from $n$ to $n+2t-1$ with $t$ slowly increasing in $n$ so the general system we would like to solve is
\begin{align*}
c_1 + c_2 + \ldots + c_{t-1} + c_t &= (2 + \epsilon_1 )t \\
c_1 c_2 + c_2 c_3 + \ldots + c_{t-2} c_{t-1} + c_{t-1} c_t &= (2+\epsilon_2) t \\
c_1c_2c_3+c_2c_3c_4+\ldots+c_{t-2}c_{t-1}c_t&=(2+\epsilon_3)t\\
& \vdots\\
c_1c_2\cdots c_{t-1}+c_2c_3\cdots c_t&=(2+\epsilon_{t-1})t\\
 c_1 c_2 \cdots c_{t-1} c_t &= (2+\epsilon_t) t.
\end{align*}
where the $\epsilon_k$'s are chosen to be small and non-zero if and only if $k \notin S$. By the same reasoning as before
$$
\lim_{n \to \infty} \frac{P_n^{(k)}}{Q_n^{(k)}} = \frac{2t}{c_1 c_2 \cdots c_k + c_2 c_3 \cdots c_{k+1} + \cdots + c_{t-k+1} c_{t-k+2} \cdots c_t}
$$
which is $1$ if and only if $k \in S$. This system is a perturbation of the system when $\epsilon_k = 0$ for all $k$ so we would like to understand the unperturbed system. Set
$$
\mathscr{R}_t:=[t,t+1] \times \bra{1, 1 + \frac{1}{t-1}}^{t-1}.
$$
We conjecture that for all $t \in \mathbb{N}$ there is a solution $(c_1, c_2, \cdots, c_t)  \in \mathscr{R}_t$. An implementation of Newton's method written in Java by the authors suggests that this holds for all $t<110$.

The following may be solved by hand.
\begin{align*}
c_1+c_2+c_3&=6\\
c_1c_2+c_2c_3&=6\\
c_1c_2c_3&=6.
\end{align*}
Let $D=3 - \sqrt{3}$ be a root of $p_3(x):=6- 6x+x^2$. Then 
\begin{align*}
c_1&= \sqrt{3} + \frac{1}{2} D + \frac{1}{2} \sqrt{D^2 - 8D + 12}\approx 3.29663;\\
c_2&=D \approx 1.26795;\\
c_3&=\sqrt{3} + \frac{1}{2}D-\frac{1}{2} \sqrt{D^2 - 8D + 12}\approx 1.43542.
\end{align*}
Clearly, $(c_1,c_2,c_3) \in \mathscr{R}_3$.

However, the situation is more complicated already when $t=4$. We solved the following system with Mathematica.
\begin{align*}
c_1 + c_2 + c_3 + c_4 &= 8\\
c_1 c_2 + c_2 c_3 + c_3 c_4 &= 8\\
c_1 c_2 c_3 + c_2 c_3 c_4 &= 8\\
c_1 c_2 c_3 c_4 &= 8.
\end{align*}
Let $D$ be the root of 
$$
p_4(x):=-512 + 7680 x^2 - 21248 x^3 + 27456 x^4 - 20544 x^5 + 9376 x^6 -  2568 x^7 + 400 x^8 - 32 x^9 + x^{10}
$$
 close to $1.15177.$
 
Then
 \begin{align*}
c_1 = &-\frac{36284}{837} -\frac{18017}{837} D + \frac{494839}{837} D^2 - \frac{1047643}{837} D^3 + \frac{8278819}{6696} D^4 \\
& - \frac{4579775}{6696} D^5 + \frac{238999}{1116} D^6 -  \frac{108529}{2976} D^7 +  \frac{41431}{13392} D^8 - \frac{1345}{13392} D^9 \approx 4.30783;\\
c_2 = & D \approx 1.15177;\\
c_3 = & \frac{9352}{837} -\frac{10469}{837} D-\frac{25058}{837} D^2 +\frac{71483}{837} D^3-\frac{147347}{1674}D^4\\
&+\frac{39620}{837} D^5-\frac{31525}{2232} D^6+\frac{214}{93} D^7-\frac{10103 }{53568}D^8 + \frac{5}{837} D^9 \approx 1.23808; \\
c_4 = & \frac{33628}{837}+\frac{27649}{837} D -\frac{469781}{837} D^2+\frac{976160}{837} D^3 -\frac{7689431 }{6696} D^4\\
&+\frac{4262815 }{6696} D^5 -\frac{446473 }{2232} D^6+\frac{101681}{2976} D^7-\frac{155621 }{53568}D^8 + \frac{1265 }{13392}D^9 \approx 1.30231.
 \end{align*}
Thus, $(c_1,c_2,c_3,c_4) \in \mathscr{R}_4$. 

Let $K$ be the splitting field of the polynomial $p_4(x)$ over $F=\mathbb{R}$.
 We used the online implementation of MAGMA to compute the Galois group of $K / F$. Thus,
$$
\hbox{Gal}(K / F) \simeq G,
$$
where $G$ is the subgroup of order $1920$ of the symmetric group $S_{10}$ generated by the set of four permutations
$$
\{ (1, 4, 3, 6, 7)(2, 10, 5, 8, 9), 
    (1, 10, 5, 8, 9)(2, 4, 3, 6, 7),
    (1, 4)(2, 10),
    (1, 10)(2, 4)\}.
$$
It is interesting that this group is a proper subgroup of $S_{10}$, but it is still not a solvable group. Thus, the roots of $p_4(x)$ may not be expressed in radicals. Since $c_2=D$ is a root of $p_4(x)$, the point $(c_1,c_2,c_3,c_4)$ may not be expressed in terms of radicals.

We remark that we can solve the system by iteratively solving each line for one of the variables and replacing every occurrence of this variable in the next lines with the expression obtained from the current line. This terminates with a polynomial in one of the variables. A solution of the system then corresponds to a root of this final polynomial with the other variables given by polynomials evaluated at this root. However these polynomials grow rapidly as $t$ increases and give little insight about the regions where these solutions lie. Thus, we feel that a purely algebraic approach is unlikely to give any useful information.


We may prove weaker versions of the theorems we desire to prove without having bounds on the solutions of these systems of equations. For example taking $(c_1, c_2, \cdots, c_t) = (t,1,\cdots,1)$ yields
$$
\lim_{t \to \infty} \frac{c_1 c_2 \cdots c_k + c_2 c_3 \cdots c_{k+1} + \cdots + c_{t-k+1} c_{t-k+2} \cdots c_t}{2t} = 1.
$$
We perturb this approximate solution by multiplying $c_i$ by a factor of $1+\epsilon$ or $\frac{1}{1+\epsilon}$ so that the first term $c_1 \cdots c_k$ is approximately $t$ or $(1+\epsilon)t$. The contribution from the rest of the terms is approximately $t$ but for general $S \subseteq \mathbb{N}$ there is an unavoidable error. However if we restrict to subsets $S$ which are almost closed under addition this error becomes negligible. While we do not achieve the desired result for every subset of $\mathbb{N}$ our results do hold for this restricted class of subsets. We believe that \reft{main} should hold for every subset of $\mathbb{N}$ and the restriction we have is an artefact of the approximation we use. One additional loss of our current approximation is that we are unable to obtain numbers that are in $\NQ \cap \NAPAbII{Q}$, but only in $\NQ \setminus \NAPII{Q}$. This will no longer be a problem if we no longer have to use this approximation. The error introduced by our approximation will appear in the proof of \refl{freq}.

We also wish to remark that it is still possible to prove stronger versions of our theorems by only considering underdetermined versions of our system. And some solutions outside of the region $\mathscr{R}_t$ may also be useful.

\section{Proof of results}

To prove \reft{diffbary} we will construct a sequence of $k$-step Markov measures on $b^\mathbb{N}$ with stationary distribution uniform over $b^k$ and with entropy converging to $\log b$. The set of generic points for these measures will be contained in $\Nk{b}{k}\setminus \Nk{b}{k+1}$ and the Hausdorff dimension of the set of generic points is equal to the entropy of the measure divided by $\log b$ by \cite{Colebrook}. Convergence of the entropy to $\log b$ then completes the proof.
\begin{proof}[Proof of \reft{diffbary}]
 Define the matrix $P_n : b^k \times b^k \to [0,1]$ such that
\begin{align*}
&P_n([0^k],[0^k]) = \frac{1+\frac{1}{n}}{b} &  &P_n([0^k],[0^{k-1} 1]) = \frac{1- \frac{1}{n}}{b} \\
&P_n([1 0^{k-1}],[0^k]) = \frac{1- \frac{1}{n}}{b} &  &P_n([1 0^{k-1}],[0^{k-1} 1]) = \frac{1+\frac{1}{n}}{b} \\
&P_n(B, B') = \begin{cases}
\frac{1}{b} \text{ if } b_2 = b'_1, \cdots, b_k = b'_{k-1} \\
0 \text{ otherwise}
\end{cases}.
\end{align*}
Then
\begin{align*}
\pr{\frac{1}{b^k}, \cdots , \frac{1}{b^k}} \times P_n = \pr{\frac{1}{b^k}, \cdots, \frac{1}{b^k}}
\end{align*}
and for each $B \in b^k$
\begin{align*}
\sum_{B' \in b^k} P_n(B,B') = 1.
\end{align*}
Thus the measure on $(b^k)^\mathbb{N}$ defined by
$$
\mu_n[B_1, \cdots, B_m] = \frac{1}{b^k} \prod_{i=1}^{m-1} P_n(B_i, B_{i+1})
$$
is shift invariant. Now by the specific form of $P_n$ the system $\pr{(b^k)^\mathbb{N}, \mu_n}$ is isomorphic to the system $\pr{b^\mathbb{N}, \nu_n}$ under the map
$$
\Phi(x)(m) = x(m)_1.
$$
Let $S_n \subseteq b^\mathbb{N}$ be the set of generic points for the measure $\nu_n$. Then $\nu_n[B] = b^{-k}$ for every block $B$ of length $k$ but $\nu_n[0^{k+1}] = \frac{1+\frac{1}{n}}{b^{k+1}}$. Thus $S_n \subseteq \Nk{b}{k} \setminus \Nk{b}{k+1}$. Furthermore by Theorem 7.2 in \cite{Colebrook}
\begin{align*}
\dim_H(S_n) &\geq \frac{h(\nu_n)}{\log b}= \frac{h(\mu_n)}{\log b} \\
&\geq \frac{- \sum_{B}\mu_n(B) \sum_{B'} P_n(B,B') \log P_n(B, B')}{ \log b}.
\end{align*}
By the continuity of the map $x \mapsto x \log x$
$$
\lim_{n \to \infty} h(\mu_n) = \log b.
$$
Thus
\begin{align*}
\dim_H(\Nk{b}{k} \setminus \Nk{b}{k+1}) \geq \lim_{n \to \infty} \dim_H(S_n) = 1.
\end{align*}
\end{proof}

Fix a set $S$ which is almost closed under addition and $\epsilon > 0$. Define 
$$
c_{t,i}:=\left\{ \begin{array}{ll}
t\cdot (1+\epsilon)^{1-\oneNSk{1}} & \textrm{if $i=1$}\\
(1+\epsilon)^{\oneNSk{i-1} - \oneNSk{i}} & \textrm{if $i>1$}.
\end{array} \right. .
$$
\begin{lem}\labl{freq}
If $S$ is almost closed under addition then
$$
\lim_{t \to \infty} \frac{\sum_{i=1}^{t-k+1} c_{t,i} c_{t,i+1} \cdots c_{t,i+k-1} }{2 t} = 1
$$
if and only if $k \in S$.

Furthermore, for any $m \geq 2$, $1 \leq r \leq m$
$$
\lim_{t \to \infty} \frac{m}{2t} \sum_{\substack{1 \leq i \leq t-m+1 \\ i \equiv r \bmod m}} c_{t,i} c_{t,i+1} \cdots c_{t,i+m-1} \neq 1
$$
and there is an integer $k$ such that
$$
\lim_{t \to \infty} \frac{m}{2t} \sum_{i=r}^{\frac{t-k+1}{m}} c_{t,i} c_{t,m+i} \cdots c_{t,km+i} \neq 1.
$$
\end{lem}
\begin{proof}
Write 
\begin{align*}
a_{t,k}  &= \sum_{i=2}^{t-k+1} \oneNSk{i-1} - \oneNSk{i+k-1};\\
b_{t,k} &= \# \{2 \leq i \leq t-k+1 : \oneNSk{i-1} - \oneNSk{i+k-1} = -1\}.
\end{align*}
Note that $|a_{t,k}| \leq 2k$
since it is a telescoping sum and $\lim_{t \to \infty} \frac{b_{t,k}}{t} = 0$ for $k \in S$
by the assumptions on $S$.
Then we have for any $k \in \mathbb{N}$
\begin{align*}
&\sum_{i=1}^{t-k+1} c_{t,i} c_{t,i+1} \cdots c_{t,i+k-1} \\
&=(1+\epsilon)^{1-\oneNSk{1} + \pr{\oneNSk{1} - \oneNSk{2}} + \cdots + \pr{\oneNSk{k-1} - \oneNSk{k}}}\cdot t\\ 
&\ \ \  +\sum_{i=2}^{t-k+1} (1+\epsilon)^{ \pr{\oneNSk{i-1} - \oneNSk{i} } + \cdots + \pr{\oneNSk{i+k-2} - \oneNSk{i+k-1}}}\\
&= (1+\epsilon)^{1-\oneNSk{k}}\cdot t + \sum_{i=2}^{t-k+1} (1+\epsilon)^{\oneNSk{i-1} - \oneNSk{i+k-1}}\\
& = (1+ \one_{\mathbb{N} \setminus S}(k) \epsilon) t + \sum_{i=2}^{t-k+1} 1 + (\oneNSk{i-1} - \oneNSk{i+k-1}) \epsilon\\
& \ \ \ + \min\{0, \oneNSk{i-1} - \oneNSk{i+k-1}\} \frac{\epsilon^2}{1+\epsilon}  \\
& = (1+ \one_{\mathbb{N}\setminus S}(k) \epsilon) t + t-k + a_{t,k} \epsilon  + b_{t,k} \frac{\epsilon^2}{1+\epsilon}
\end{align*}
where we have used that $\frac{1}{1+\epsilon} = 1-\epsilon + \frac{\epsilon^2}{1+\epsilon}.$
Now if $k \in S$ this is simply
\begin{align*}
2t + o(t)
\end{align*}
and if $k \notin S$ this is bounded below by
\begin{align*}
(2+\epsilon)t + o(t)
\end{align*}
since $b_{t,k} \geq 0$. Thus the first claim holds.

For the next two claims consider first $r = 1$. We find by the same estimates
\begin{align*}
\sum_{\substack{1 \leq i \leq t-m+1 \\ i \equiv 1 \bmod m}} c_{t,i} c_{t,i+1} \cdots c_{t,i+m-1} &= t \cdot (1+\epsilon)^{1-\oneNSk{m}} + \sum_{\substack{m+1 \leq i \leq t-m+1 \\ i \equiv 1 \bmod m}} (1+\epsilon)^{\oneNSk{i-1} - \oneNSk{i+m-1}} \\
& \geq t \cdot(1 + \frac{1}{(1+\epsilon)m} + \one_{\mathbb{N} \setminus S}(m) \epsilon) + o(t).
\end{align*}
For $r \neq 1$ the same estimates give
\begin{align*}
\sum_{\substack{1 \leq i \leq t-m+1 \\ i \equiv r \bmod m}} c_{t,i} c_{t,i+1} \cdots c_{t,i+m-1} & = \sum_{\substack{1 \leq i \leq t-m+1 \\ i \equiv r \bmod m}} (1+\epsilon)^{\oneNSk{i-1} - \oneNSk{i+m-1}}\\
&\leq (1+\epsilon)\frac{t}{m} + o(t).
\end{align*}
Similarly
\begin{align*}
\sum_{i=1}^{\frac{t-k+1}{m}} c_{t,1+mi} c_{t,1+m(i+1)} \cdots c_{t,1+m(i+k-1)} \geq (1+\epsilon)^{-k} t \pr{1 + \frac{1}{m}} + o(t)
\end{align*}
and for $r \neq 1$
\begin{align*}
\sum_{i=r}^{\frac{t-k+1}{m}} c_{t,r+mi} c_{t,r+m(i+1)} \cdots c_{t,1+m(i+k-1)} \leq (1+\epsilon)^k \frac{t}{m} + o(t).
\end{align*}
We can take $\epsilon$ sufficiently small and $k= 2$ to complete the proof.
\end{proof}

We need the following lemma.

\begin{lem}\labl{fulldivAPdiv}
If $Q$ is fully divergent and non-decreasing then $Q$ is $(m,r)$-divergent of type $I$ and $(k,m,r)$-divergent of type $II$ for every $k,m,r$.
\end{lem}

\begin{proof}
If $Q$ is fully divergent and non-decreasing then $q_n = o(n^a)$ for every $a > 0$. Then
\begin{align*}
\lim_{n \to \infty} \frac{\sum_{i \equiv r \bmod m}^n \frac{1}{q_i q_{i+1} \cdots q_{i+k-1} }}{\sum_{i \equiv r \bmod m}^n \frac{1}{i^a (i+1)^a \cdots (i+k-1)^a}} \geq 1.
\end{align*}
Taking $a < \frac{1}{k}$ causes the denominator to grow to $\infty$ which implies the numerator does as well. Thus $Q$ is $(m,r)$-divergent of type $I$ for all $m,r$. The same argument applies to sums of the form
$$
\sum_{i \equiv r \bmod m} \frac{1}{q_i q_{i+m} \cdots q_{i+(k-1)m}}
$$
which implies $Q$ is also $(k,m,r)$-divergent of type $II$.
\end{proof}

\begin{proof}[Proof of \reft{main}]
Write $Q = [2]^{\ell_2} [3]^{\ell_3} [4]^{\ell_4} \cdots$.  We may write $Q$ in this form as it is monotone.
For $i \in \mathbb{N}$ define 
\begin{align*}
L_i &= \sum_{j=2}^i \ell_j;\\
\alpha(n) &= \min \{ j  : L_j \geq n \}.
\end{align*}
Since $Q$ is fully divergent, we have $q_n = o(n^a)$ for all $a>0$. Now $n \leq L_{q_n}$ so $q_n = o(L_{q_n}^a)$ which implies $\lim_{b \to \infty} \frac{L_b}{b^a} = \infty$ for all $a > 0$. Define the sequence $t_i = i!$.
We will now define the sequences $(\kappa_i)$ and $(K_i)$. Set $K_0=0$. For $i \in \mathbb{N}$ let $\kappa_i$ be defined to be the smallest integer $j$ satisfying the conditions
\begin{enumerate}
\item $q_n \geq (1+\epsilon) t_{i+1}^2$  for all $n > K_{i-1} + 2j t_i$; \label{cond1}
\item $2 t_{i+1} | K_{i-1} + 2j t_i $; \label{cond2}
\item $q_n < n^{\frac{1}{i+1}}$ for $n > K_{i-1} + 2 j t_i$;
\item $K_{i-1} + 2t_i j > i (2t_{i+1})^{\frac{1}{i+1}}$.
\end{enumerate}
Finally, put 
$$
K_i= \max\{K_{i-1} + 2\kappa_i t_i, t_{i+1}^2\}
$$ 
for positive integers $i$. Define $i(n)$ to be the unique integer such that $K_{i(n)-1} < n \leq K_{i(n)}$ and $j(n)$ the unique integer such that 
$$
K_{i(n)-1} + 2 t_{i(n)} j(n) < n \leq K_{i(n)-1}  + 2t_{i(n)} (j(n)+1).
$$ 
Note that $n > K_{i(n)-1} \geq t_{i(n)}^2$.

Define the basic sequence $P = (p_n)$ as follows
$$
p_n = \Xi(Q,t_{i(n)})_n.
$$
where
$$
\Xi(Q,t)_n = \begin{cases}
\max\{\lfloor c_{t,r}^{-1} q_n \rfloor, 2\} & \text{ if } n \equiv r \bmod 2t \text{ for } 1 \leq r \leq t \\
2^n q_n & \text{ if } n \bmod 2t > t.
\end{cases} 
$$
Since $q_{K_i} > t_i^2$ for all $i$ by  \ref{cond1}, we have
$$1-\frac{1}{t_i} < \frac{\floor{c_{t_i,r}^{-1} q_n}}{c_{t_i,r}^{-1} q_n} < 1$$
for each $K_{i-1} < n \leq K_i$ and $1 \leq r \leq t_i$.
For $i,j \in \mathbb{N}$ define the sums
\begin{align*}
P_{i,j}^{(k)} &= \sum_{v=1}^{2t_i} \frac{1}{p_{K_i + 2t_i j + v} p_{K_i + 2 t_i j + v+1} \cdots p_{K_i + 2 t_i j+ v + k -1}};\\
Q_{i,j}^{(k)} &= \sum_{v=1}^{2t_i} \frac{1}{q_{K_i + 2t_i j+ v} q_{K_i + 2 t_i j+ v+1} \cdots q_{K_i + 2t_i j+ v+ k-1}}
\end{align*}
where $0 \leq j < \kappa_i$. 
Note that
\begin{align*}
\lim_{n \to \infty} &\frac{\pnk}{\pr {\sum_{a=1}^{i(n)-1} \sum_{b=0}^{\kappa_a} P_{a,b}^{(k)}} + \sum_{b=0}^{j(n)-1} P_{i(n),b}^{(k)}} = 1;\\
\lim_{n \to \infty} &\frac{\qnk}{\pr {\sum_{a=1}^{i(n)-1} \sum_{b=0}^{\kappa_a} Q_{a,b}^{(k)}} + \sum_{b=0}^{j(n)-1} Q_{i(n),b}^{(k)}} = 1
\end{align*}
since
$$
\lim_{n \to \infty} \frac{t_{i(n)}}{Q_{n}^{(k)}} \leq \lim_{n \to \infty} \frac{t_{i(n)}}{n / q_n^k} = \lim_{n \to \infty} \frac{q_n^k t_{i(n)}}{n} \leq \lim_{n \to \infty} \frac{q_n^k}{n^{1/3}} \frac{t_{i(n)}}{t_{i(n)}^{4/3}} = 0
$$
with a similar inequality holding for $P_n^{(k)}$.
Define the sets 
\begin{align*}
U_n = \{ (a,b) \in \mathbb{N}^2 : 0 \leq b < \kappa_a, 0\leq a < i(n) \text{ or } a = i(n), 0 \leq b \leq j(n) \\ \text{ and } q_k = q_\ell \text{ for all } K_a +2bt_a + 1 \leq k<\ell \leq K_a+2(b+1)t_a\}.
\end{align*}
Note that 
$$
K_{i(n)-1} + 2 j(n) t_{i(n)} - \# U_n  \leq q_{K_{i(n)-1} + 2 j(n) t_{i(n)}}
$$ 
since there are at most $q_{n} \leq q_{K_{i(n)-1} + 2j(n) t_{i(n)}}$ positions where the basic sequence increases up to position $n$ by monotonocity. Thus we have
\begin{align*}
\lim_{n \to \infty} \frac{K_{i(n)-1} + 2 j(n) t_{i(n)} - \# U_n}{K_{i(n)-1} + 2(j(n)-1) t_{i(n)}} & \leq \lim_{i \to \infty}\max_{1 \leq m \leq \kappa_{i}} \frac{q_{K_{i-1} + 2 m t_i}}{K_{i-1} + 2 (m-1) t_i} \\
&  \leq \lim_{i \to \infty} \max \br{ \frac{q_{K_{i-1} + 2t_i}}{K_{i-1}}, \max_{2 \leq m \leq \kappa_i} \frac{q_{K_{i-1} + 2 m t_i}}{K_{i-1} + 2 (m-1) t_i}} \\
&  \leq \lim_{i \to \infty} \max \br{ \frac{(K_{i-1} + 2t_i)^{1/i}}{K_{i-1}}, \max_{2 \leq m \leq \kappa_i} \frac{(K_{i-1} + 2 m t_i)^{1/i}}{K_{i-1} + 2 (m-1) t_i}} \\
&  \leq \lim_{i \to \infty} \max \br{ \frac{K_{i-1}^{1/i} + 2t_i^{1/i}}{K_{i-1}}, \max_{2 \leq m \leq \kappa_i} \frac{K_{i-1}^{1/i} + (2 m t_i)^{1/i}}{K_{i-1} + 2 (m-1) t_i}} \\
&  \leq \lim_{i \to \infty} \max \br{ K_{i-1}^{1/i-1} + \frac{2t_i^{1/i}}{t_{i}^2}, \max_{2 \leq m \leq \kappa_i} \frac{K_{i-1}^{1/i}}{K_{i-1}} + \frac{(2 m t_i)^{1/i}}{2 (m-1) t_i}} \\
&  \leq \lim_{i \to \infty} K_{i-1}^{1/i-1} + \max \br{  \frac{2t_i^{1/i}}{t_{i}^2},  \frac{2 t_i^{1/i}}{t_i}} \\
& = 0.
\end{align*}
In particular this estimate implies
\begin{align*}
\lim_{n \to \infty} \frac{\sum_{(a,b) \in U_n} P_{a,b}^{(k)} + \sum_{(a,b) \notin U_n} P_{a,b}^{(k)}}{\sum_{(a,b) \in U_n} P_{a,b}^{(k)}} = 1
\end{align*}
and
\begin{align*}
\lim_{n \to \infty} \frac{\sum_{(a,b) \in U_n} Q_{a,b}^{(k)} + \sum_{(a,b) \notin U_n} Q_{a,b}^{(k)}}{\sum_{(a,b) \in U_n} Q_{a,b}^{(k)}} = 1.
\end{align*}
Thus
\begin{align*}
&\lim_{n \to \infty} \frac{P_n^{(k)}}{Q_n^{(k)}} = \lim_{n \to \infty} \frac{\sum_{a=1}^{i(n)-1} \sum_{b=0}^{\kappa_a} P_{a,b}^{(k)} + \sum_{b=0}^{j(n)-1} P_{i(n),b}^{(k)}}{\sum_{a=1}^{i(n)-1} \sum_{b=0}^{\kappa_a} Q_{a,b}^{(k)} + \sum_{b=0}^{j(n)-1} Q_{i(n),b}^{(k)}}  \\
&= \lim_{n \to \infty} \frac{\sum_{(a,b) \in U_n} P_{a,b}^{(k)} +  \sum_{(a,b) \notin U_n} P_{a,b}^{(k)} }{\sum_{(a,b) \in U_n} Q_{a,b}^{(k)} + \sum_{(a,b) \notin U_n} Q_{a,b}^{(k)}} \\
& = \lim_{n \to \infty} \frac{\sum_{a,b \in U_n} P_{a,b}^{(k)}}{\sum_{a,b \in U_n} Q_{a,b}^{(k)}} \\
& \geq \lim_{n \to \infty} \frac{\sum_{a=1}^{i(n)-1} \sum_{b = 1}^{\kappa_a} \pr{ (1-\frac{1}{t_a})^k \frac{\sum_{r=1}^{t_a-k+1} c_{t_a,r} c_{t_a,r+1} \cdots c_{t_a,r+k-1}}{q_{K_{a-1} + 2 b t_a}^k} + \frac{t_a}{2^{K_{a-1} + 2 t_a b} q_{K_{a-1} + 2bt_a}^k } } }{\sum_{a=1}^{i(n)-1} \sum_{b=1}^{\kappa_a} \frac{2 t_a}{q_{K_{a-1} + 2 b t_a}^k} + \sum_{b=0}^{j(n)-1} \frac{2 t_{i(n)}}{q_{K_{i(n)-1} + 2 b t_{i(n)}}^k} }\\
& \ \ \ \ \ + \frac{\sum_{b=0}^{j(n) - 1}  \pr{ (1-\frac{1}{t_{i(n)}})^k \frac{\sum_{r=1}^{t_{i(n)}-k+1} c_{t_{i(n)},r} c_{t_{i(n)},r+1} \cdots c_{t_{i(n)},r+k-1}}{q_{K_{i(n)-1} + 2 b t_{i(n)}}^k} + \frac{t_{i(n)}}{2^{K_{i(n)-1} + 2bt_{i(n)}} q_{K_{i(n)-1} + 2bt_{i(n)}}^k } }}{\sum_{a=1}^{i(n)-1} \sum_{b=1}^{\kappa_a} \frac{2 t_a}{q_{K_{a-1} + 2 b t_a}^k} + \sum_{b=0}^{j(n)-1} \frac{2 t_{i(n)}}{q_{K_{i(n)-1} + 2 b t_{i(n)}}^k}} \\
&= \lim_{t \to \infty} \frac{\sum_{r=1}^{t-k+1} c_{t,r} c_{t,r+1} \cdots c_{t,r+k-1} }{2t}.
\end{align*}
Thus
$$
\lim_{n \to \infty} \frac{\pnk}{\qnk} = 1
$$
if and only if $k \in S$. 

Now fix $m \in \mathbb{N}$ and $0 \leq r \leq m-1$. Since $m | K_i$ and $m | t_i$ for sufficiently large $i$ we have that for all $a,a'$ large enough and $0 \leq b < \kappa_a$, $0 \leq b' < \kappa_{a'}$ we have $K_{a-1} + 2 b t_a \equiv K_{a'-1} + 2b' t_{a'} \bmod m$. Then by similar calculations and \refl{freq} we have
\begin{align*}
\lim_{n \to \infty} \frac{P_n^{(m,r)}}{Q_{n}^{(m,r)}} &\neq 1 \\
\end{align*}
for all $k$. Similarly by \refl{freq} we have
\begin{align*}
\lim_{n \to \infty} \frac{\sum_{i \equiv r \bmod m}^n \frac{1}{p_i p_{i+m} \cdots p_{i+(k-1)m}}}{\sum_{i \equiv r \bmod m}^n \frac{1}{q_i q_{i+m} \cdots q_{i+(k-1)m}}} \neq 1
\end{align*}
for some $k \geq 2$. Note that this sequence is bounded and bounded away from $0$. Now $Q$ is monotone and fully divergent so by \refl{fulldivAPdiv} it is also $(m,r)$-divergent of type $I$ and $(k,m,r)$-divergent of type $II$ for all $k,m,r$, thus $P$ is as well. Therefore $\N{P} \cap \NAPI{P} \cap \NAPII{P}$ is non-empty.
Since 
\begin{align*}
N_n^Q(B,\Psi_{P,Q}(x)) &= N_n^P(B,x) + O(1), \\
N_{n,m,r}^Q(B,\Psi_{P,Q}(x)) &= N_{n,m,r}^P(B,x) + O(1), \hbox {and}\\ 
N_{n,m,r}^Q\prime(B,\Psi_{P,Q}(x)) &= N_{n,m,r}^P \prime(B,x) + O(1)
\end{align*}
for all $x \in \N{P} \cap \NAPI{P} \cap \NAPII{P}$ then 
$$
\Psi_{P,Q}(x) \in \RNAPI{Q} \cap \RNAPII{Q} \cap \NAPAbI{Q} \cap  \bigcap_{k \in S} \Nk{Q}{k} \setminus\pr{ \NAPII{Q} \cup \bigcup_{k \notin S} \Nk{Q}{k} }.
$$
\end{proof}

We will need the following theorem of D. Feng, Z. Wen, and J. Wu from \cite{FengWenWu} to prove \reft{HDmain}.

 Let $( n_k )$ be a sequence of positive integers and $(c_k)$ be a sequence of positive numbers such that $n_k \geq 2$, $0<c_k<1$, $n_1 c_1 \leq \delta$, and $n_k c_k \leq 1$, where $\delta$ is a positive real number. For any $k$, let $D_k = \{ (i_1, \cdots, i_k): 1\leq i_j \leq n_j, 1\leq j \leq k \}$, and $D = \bigcup D_k$, where $D_0 =\emptyset$. If $\sigma = ( \sigma_1, \cdots , \sigma_k) \in D_k$, $\tau = (\tau_1 ,\cdots , \tau_m) \in D_m$, put $\sigma * \tau = (\sigma_1, \cdots , \sigma_k, \tau_1, \cdots , \tau_m)$.

\begin{definition}
Suppose $J$ is a closed interval of length $\delta$. The collection of closed subintervals $ \mathcal{F} = \{ J_\sigma : \sigma \in D\}$ of $J$ has \textit{homogeneous Moran structure} if:
\begin{enumerate}
	\item $J_{\emptyset} = J$;
	\item $\forall k \geq 0, \sigma \in D_k, J_{\sigma *1}, \cdots , J_{\sigma * n_{k+1}}$ are subintervals of $J_\sigma$ and $\mathring{J}_{\sigma*i}\cap \mathring{J}_{\sigma*j}=\emptyset$ for $i \neq j$;
	\item $\forall k \geq 1, \forall \sigma \in D_{k-1}, 1\leq j \leq n_k$, $c_k = \frac{\lambda(J_{\sigma*j})}{\lambda(J_\sigma)}$.
\end{enumerate}
\end{definition}

Suppose that $\mathcal{F}$ is a collection of closed subintervals of $J$ having homogeneous Moran structure. Let $E(\mathcal{F}) = \bigcap_{k\geq 1} \bigcup_{\sigma \in D_k} J_\sigma$. We say $E(\mathcal{F})$ is a \textit{homogeneous Moran set determined by} $\mathcal{F}$, or it is a \textit{homogeneous Moran set determined by} $J$, $( n_k )$, $( c_k )$. 

\begin{thrm}\labt{moran}[D. Feng, Z. Wen, and J. Wu]
If $S$ is a homogeneous Moran set determined by $J$, $(n_k )$, $( c_k )$, then
\begin{equation*}
\liminf_{k \to \infty} \frac{\log n_1 n_2 \cdots n_k}{-\log c_1 c_2 \cdots c_{k+1} n_{k+1}} \leq \dimh{S} \leq \liminf_{k \to \infty} \frac{\log n_1 n_2 \cdots n_k}{-\log c_1 c_2 \cdots c_k}.
\end{equation*}
\end{thrm}

\begin{proof}[Proof of \reft{HDmain}]
Let $P = (p_n) = [2]^{\ell_2} [3]^{\ell_3} \cdots$ be infinite in limit, fully divergent, and monotone. By \reft{main} there is an
 $$
\xi \in  \RNAPI{Q} \cap \RNAPII{Q} \cap \NAPAbI{Q} \cap  \bigcap_{k \in S} \Nk{Q}{k} \setminus\pr{ \NAPII{Q} \cup \bigcup_{k \notin S} \Nk{Q}{k} }.
$$ 
Write $\xi = F_0. F_1 F_2 \cdots \wrt{P}$. Now define
\begin{align*}
Q &= [2]^{\ell_2} [2^{2^2} \cdot 2]^{2^2 \ell_2} [3]^{\ell_3} [2^{2^3} \cdot 3]^{2^3 \ell_3} [4]^{\ell_4} [2^{2^4} \cdot 4]^{2^4 \ell_4} \cdots [n]^{\ell_n} [2^{2^n} \cdot n]^{2^n \ell_n} \cdots;\\
L_i &= \sum_{j=2}^{i+1} \ell_j;\\
M_i &= \sum_{j=2}^{i+1} (2^j+1)\ell_j;\\
\alpha(n) &= \min \{j : M_j < n\};\\
\epsilon_n&= \frac{\min\{\log q_1 \cdots q_{n-1}, \log q_n\}^{1/2}}{\log q_n};\\
\omega_n &= q_n^{1-\epsilon_n}.
\end{align*}
Now define the sets
\begin{align*}
V_n &= \begin{cases}
\{F_{L_{\alpha(n)} + n-M_{\alpha(n)}}\} &\text{ if } n - M_{\alpha(n)} \leq \ell_{\alpha(n)} \\
\left [ \max\{\floor{q_n x_n }, \log \alpha(n) \} - \omega_n,  \max\{\floor{q_n x_n }, \log \alpha(n) \} + \omega_n\right ] &\text{ otherwise}
\end{cases};\\
\Omega &= \{ y = .E_1 E_2 \cdots \wrt{Q} : E_n \in V_n \text{ for all } n\}
\end{align*} 
and let $y = .E_1 E_2 \cdots \wrt{Q} \in \Omega$.
Note that $\lim_{n \to \infty} \frac{E_n}{q_n} - \frac{\floor{q_n x_n}}{q_n} = 0$ since $\frac{\log \alpha(n)}{q_n} \to 0$ and $\frac{\omega_n}{q_n} \to 0$. This implies $G((T_{Q,n}(x) - x_n)_n)$ consists of the single function $1- \one_{\{0\}}$ since the difference between the two sequences tends to $0$ except along a set of density $0$. Fix a block $B$ of length $k$ and let $m$ be the maximum digit in $B$. Because $\alpha(n) \to \infty$ we must have that $m < \alpha(n)$ for $n$ sufficiently large, so if we define 
$$
g(n) =  L_{\alpha(n)} + \min\{n-M_{\alpha(n)}, \ell_{\alpha(n)}\}
$$ 
we have
$$
N_n^Q(B,y) = N_{g(n)}^P(B,x) + O(1).
$$
On the other hand, we have
\begin{align*}
\lim_{n \to \infty} \frac{P_{g(n)}^{(k)}}{Q_n^{(k)}} &= \lim_{n \to \infty} \frac{\sum_{i=1}^{g(n)} \frac{1}{p_i \cdots p_{i+k-1}}}{\sum_{i=1}^n \frac{1}{q_i \cdots q_{i+k-1}}}  \\
& = \lim_{m \to \infty}\frac{\sum_{i=1}^{m} \frac{\ell_i}{i^k}}{\sum_{i=1}^m \pr{\frac{\ell_i}{i^k} + \frac{(2^i+1) \ell_i}{2^{k 2^i} i^k}}} = 1.\\
\end{align*}
Similar estimates holds for the corresponding quantities for $Q$ AP-normality of types $I$ and $II$.

Thus $y \in \mathscr{N}_k(Q)$ if and only if $\xi \in \Nk{P}{k}$. Furthermore $y \in \NAPAbI{Q} \setminus  \NAPII{Q}$. This implies 
$$
\Omega \subset \DistQ \cap \RNAPI{Q} \cap \RNAPII{Q} \cap \NAPAbI{Q} \cap \bigcap_{k \in S} \Nk{Q}{k} \setminus \pr{\bigcup_{k \notin S} \Nk{Q}{k} \cup \NAPII{Q}}.
$$ 
Furthermore, by \reft{moran}
\begin{align*}
\dim_H(\Omega)& \geq \lim_{n \to \infty} \frac{\log \omega_1 \cdots \omega_n}{-\log \frac{1}{q_1} \cdots \frac{1}{q_{n}}\omega_{n+1}} \\
& = \lim_{n \to \infty} \frac{\sum_{i=1}^n (1-\epsilon_i) \log q_i}{\sum_{i=1}^{n+1} \log q_i - (1-\epsilon_{n+1}) \log q_{n+1}} \\
& = \lim_{n \to \infty} \frac{1}{1+ \epsilon_{n+1} \frac{\log q_{n+1}}{\log q_1 \cdots q_n}} = 1.
\end{align*}
Thus
$$
\dim_H\left (\DistQ \cap \RNAPI{Q} \cap \RNAPII{Q} \cap \NAPAbI{Q} \cap \bigcap_{k \in S} \Nk{Q}{k} \setminus \pr{\bigcup_{k \notin S} \Nk{Q}{k} \cup \NAPII{Q}}\right ) = 1.
$$
\end{proof}

\begin{proof}[Proof of \reft{comp}]
Let $\epsilon$ be a sufficiently small rational. For $i<6$ put $\ell_i = 0$ and for $i \geq 6$ put $\ell_{i} = 3^{i!} \cdot (i+1)^{i! i}$. Then $P$ is a computable basic sequence. Following the proof of \reft{main}, note that the sequences $(L_n)$, $(\alpha(n))$, $(t_i)$, $(\kappa_i)$, $(K_i)$, $(i(n))$, $(j(n))$, and $(\Xi(P,t_{i(n)})_n)$ are computable. Thus $Q = (q_n) = (\Xi(P,t_{i(n)}))_n$ is a computable basic sequence. By Theorem 1.8 in \cite{ppq2} there is a real number $x = 0.E_1 E_2 \wrt{P}$ in $\N{P} \cap \NAPI{P} \cap \NAPII{P}$ with $(E_n)$ a computable sequence. Thus $\Psi_{P,Q}(x)$ is a computable real number since its sequence of digits is $\min\{E_n,q_n-1\}$ which is computable. Therefore we have a computable real number in 
$$
\RNAPI{Q} \cap \RNAPII{Q} \cap \NAPAbI{Q} \cap \bigcap_{k \in S} \Nk{Q}{k} \setminus \pr{\NAPII{Q} \cup \bigcup_{k \notin S} \Nk{Q}{k}}.
$$
\end{proof}

\bibliographystyle{amsplain}

\providecommand{\bysame}{\leavevmode\hbox to3em{\hrulefill}\thinspace}
\providecommand{\MR}{\relax\ifhmode\unskip\space\fi MR }
\providecommand{\MRhref}[2]{%
  \href{http://www.ams.org/mathscinet-getitem?mr=#1}{#2}
}
\providecommand{\href}[2]{#2}

\end{document}